\documentclass[12pt,pdftex,noinfoline]{imsart}

\usepackage[usenames,dvipsnames]{xcolor}
\usepackage{natbib}
\bibpunct{(}{)}{;}{a}{,}{,}

\usepackage[pdftex,pdfborderstyle={/S/U/W 1},colorlinks]{hyperref}
\RequirePackage[OT1]{fontenc}
\usepackage[linesnumbered]{algorithm2e}
\RequirePackage{amsthm,amsmath,amsfonts}
\RequirePackage{hypernat}
\usepackage{fullpage}
\usepackage{graphicx}
\usepackage{hyperref}
\usepackage{amsthm, amssymb, amsmath}
\usepackage{url}
\usepackage[english]{babel}
\usepackage{times}
\usepackage[T1]{fontenc}
\usepackage{bbm}
\usepackage{xspace}
\usepackage{rotating,multirow}
\usepackage{enumerate}
\usepackage{mathtools}
\usepackage{tikz}

\def\interleave{|\kern-.25ex|\kern-.25ex|}
\def\interleavesub{|\kern-.15ex|\kern-.15ex|}

\newcommand{\nNorm}[1]{\left|\kern-.25ex\left|\kern-.25ex\left| {#1}\right|\kern-.25ex\right|\kern-.25ex\right|}

\newcommand{\E}{{\mathbb E}}

\newcommand{\eps}{{\epsilon}}

\newcommand{\R}{{\mathbb R}}

\renewcommand{\P}{{\mathbb P}}

\newcommand{\A}{{\mathcal{A}}}
\newcommand{\M}{{\mathcal{M}}}

\newcommand{\C}{{\mathcal{C}}}

\newcommand{\I}{{\mathcal{I}}}

\newcommand{\U}{{\mathcal{U}}}

\newcommand{\F}{{\cal F}}

\newcommand{\bt}{\theta}

\newcommand{\s}{s}

\def\min{\mathop{\text{\rm min}}}
\def\max{\mathop{\text{\rm max}}}
\def\inf{\mathop{\text{\rm inf}}}
\def\sup{\mathop{\text{\rm sup}}}

\numberwithin{equation}{section}
\theoremstyle{plain}
\newtheorem{theorem}{Theorem}[section]

\newtheorem{lemma}{Lemma}[section]
\newtheoremstyle{remark}{\topsep}{\topsep}%
     {\normalfont}
     {}           
     {\bfseries}  
     {.}          
     {.5em}       
     {\thmname{#1}\thmnumber{ #2}\thmnote{ #3}}
\theoremstyle{remark}
\newtheorem{remark}{Remark}[section]

\long\def\comment#1{}

\def\P{{\mathbb P}}
\def\E{{\mathbb E}}
\def\supp{\mathop{\text{supp}\kern.2ex}}
\def\argmin{\mathop{\text{\rm arg\,min}}}

\def\argmax{\mathop{\text{\rm arg\,max}}}
\let\hat\widehat
\let\tilde\widetilde

\let\hat\widehat
\let\tilde\widetilde

\def\1{{(1)}}
\def\2{{(2)}}

\def\M{{\mathcal{M}}}

\long\def\comment#1{}

\long\def\comment#1{}

\def\P{{\mathbb P}}
\def\E{{\mathbb E}}

\def\supp{\mathop{\text{supp}\kern.2ex}}
\def\argmin{\mathop{\text{\rm arg\,min}}}

\def\argmax{\mathop{\text{\rm arg\,max}}}
\let\tilde\widetilde

\let\hat\widehat
\let\tilde\widetilde

\def\1{{(1)}}
\def\2{{(2)}}

\long\def\comment#1{}

\def\threebars{\mbox{$|\kern-.25ex|\kern-.25ex|$}}

\def\F{\mathbb{F}}

\begin{document}

\hypersetup{citecolor=MidnightBlue}
\hypersetup{linkcolor=Black}
\hypersetup{urlcolor=MidnightBlue}

\begin{frontmatter}

\mbox{}
\vskip.25in
\centerline{\Large\bf Adaptive Risk Bounds in Unimodal Regression}
\runtitle{Risk Bounds in Unimodal Regression}

\begin{aug}
\author{\fnms{Sabyasachi} \snm{Chatterjee}\ead[label=e1]{sabyasachi@galton.uchicago.edu}}
\and
\author{\fnms{John} \snm{Lafferty}\ead[label=e4]{lafferty@galton.uchicago.edu}}
\address{
\vskip1pt
\begin{tabular}{c}
Department of Statistics \\
The University of Chicago 
\end{tabular}
\\[10pt]
\today\\[5pt]
}
\end{aug}

\begin{abstract}
We study the statistical properties of the least squares
estimator in unimodal sequence estimation.  Although closely related to
isotonic regression, unimodal regression has not been as extensively
studied. We show that the unimodal least squares estimator 
is adaptive in the sense that the risk scales 
as a function of the number of values in the true underlying
sequence. Such adaptivity properties have been shown for isotonic regression by
\cite{chatterjee2015risk} and \cite{bellec:2016}. A technical
complication in unimodal regression is the non-convexity of
the underlying parameter space. We develop a general variational
representation of the risk that holds whenever the parameter space can
be expressed as a finite union of convex sets, using techniques that
may be of interest in other settings.
\end{abstract}

\vskip15pt
\end{frontmatter}

\maketitle

\section{Introduction}
In this paper we investigate the statistical properties of the least
squares estimator in unimodal sequence estimation, where the sequence
rises to a mode, and then decreases.  Unimodal regression is a natural
type of shape constrained inference problem. Although closely related
to isotonic regression, unimodal regression has not been as
extensively studied.  We analyze the least squares estimator for
unimodal regression, showing that the estimator is adaptive in the
sense that the risk scales as a function of the number of values in
the true underlying sequence. When the sequence has a relatively small
number of values, the estimator achieves essentially parametric rates
of convergence.  Such adaptivity properties have been shown for
isotonic regression in the recent literature
\citep{chatterjee2015risk,bellec:2016}. However, existing proof
techniques for isotonic regression do not directly extend to the
unimodal setting---a technical complication is the non-convexity of
the underlying parameter space. We develop a general variational
representation of the risk that holds whenever the parameter space can
be expressed as a finite union of convex sets, and employ empirical
process techniques that give good upper bounds.  These techniques enable us
to show that the least squares estimator for unimodal regression, to a
considerable extent, enjoys similar adaptivity properties as the least
squares estimator for isotonic regression.

In more detail, we consider the problem of unimodal regression where for design points $x_1 \leq x_2 \dots \leq x_n$ we observe 
\begin{equation*}
y_{i} = f(x_i) + z_{i}, \qquad \mbox{for } i=1,\ldots,
n
\end{equation*}
where $f: \R \rightarrow \R$ is a unimodal function and the random errors $z_{i}$ are assumed to be independently distributed as $N(0,\sigma^2)$, with the variance $\sigma^2$ unknown. Our analysis also carries over to the case when $z$ is a mean zero error vector with independent entries and absolute value bounded by a constant $\sigma > 0.$ We consider the design points to be fixed but arbitrary, and hence the problem of estimating the unimodal function $f$ reduces to the problem of estimating an unknown vector $\theta^* \in \R^{n}$ from observations
\begin{equation}\label{eq:RegMdl}
y_{i} = \theta_{i}^* + z_{i}, \qquad \mbox{for } i=1,\ldots,
n
\end{equation}
where $\theta^*$ is constrained to lie in
\begin{equation}\label{eq:M}
 \mathcal{U}_n := \{ \theta \in \R^{n}: \theta_{1} \leq \theta_{2} \leq \dots \leq  \theta_{m} \geq \theta_{m + 1} \dots \geq \theta_{n} \:\:\mbox{for some}\:\: 1 \leq m \leq n\}.
\end{equation}
We refer to any vector in $\U_{n}$ as a
\textit{unimodal sequence}; such a sequence first increases
and then decreases. 

In this paper we are
concerned with the statistical problem of estimating $\theta^*\in\U_n$
from the data $y$. This problem has interesting structure that we shall exploit.
For any $1 \leq m \leq n,$ let us define the set
\begin{equation}
C_m = \{\theta \in \R^{n}: \theta_{1} \leq \theta_{2} \leq \dots \leq  \theta_{m} \geq \theta_{m + 1}, \dots \geq \theta_{n}\},
\end{equation}
thus using the index $m$ to indicate the mode of the sequence.  Each
set $C_m \subset \R^n$ is a closed convex cone.  When $m= 1$, the set $C_{1}$ is the collection of 
decreasing sequences; when $m = n$, the set
$C_{n}$ is the collection of increasing sequences. Since
$\U_n = \bigcup_{m = 1}^{n} C_m$, we see that the collection of
unimodal sequences can be written as a union of
$n$ closed convex cones, but is not itself a convex set.

The least squares estimator (LSE) $\hat{\theta}$ is defined as the
minimizer of the squared Euclidean norm, $\|y - \theta\|^2$, over
$\theta \in \U_n$,
\begin{equation}\label{eq:LSE}
\hat{\theta} := \argmin_{\theta \in \U_n} \|y - \theta\|^2.   
\end{equation}
Since $\U_n$ is not a convex
set, the LSE $\hat \theta$ may not be uniquely defined. In case there
are multiple minima, the LSE $\hat{\theta}$ can be chosen arbitrarily
over the set of minimizers. For each $1 \leq m \leq n$ since $C_m$
is a closed convex cone, we can define $\hat{\theta}^m$ as the unique projection
of $y$ onto $C_m$. Then the LSE can also be written
as
\begin{equation}\label{union}
\hat{\theta} := \argmin_{\theta \in \{\hat{\theta}^1,\dots,\hat{\theta}^n\}} \|y - \theta\|^2.
\end{equation}

The problem of computing a projection to
the set of unimodal sequences has received considerable
attention \citep{stout2000optimal,stout2008unimodal,chemometrics1998least,boyarshinov2006linear}. 
\cite{stout2008unimodal} shows that by using variations of
the well-known pooled adjacent violators (PAVA) algorithm
for isotonic regression \citep{BBBB72,grotzinger1984projections}, one
can design an $O(n)$ algorithm to
compute a unimodal projection, and hence the LSE.
Thus, there is effectively no 
computational price to pay when fitting unimodal
sequences rather than monotone sequences.
At a high level, the contribution
of the current paper is to show that there is also essentially no
statistical price to pay.

Compared with monotone or isotonic regression, the unimodal regression
problem has received relatively little attention in the statistics
literature.
\cite{frisen1986unimodal} presents several
applications of unimodal regression and introduces the least squares
estimator, without analyzing its risk properties.
\cite{shoung2001least} study the convergence of the
mode of the LSE as an estimator of the true mode of a unimodal
function. For classes of unimodal functions indexed by a smoothness
parameter, the authors prove rates of convergence of the mode of the
LSE and also show that the rates are minimax optimal up to logarithmic
factors, for a given smoothness class.  
\cite{kollmann2014unimodal} propose the use of a penalized
estimator based on splines to estimate the underlying unimodal
function, but without studying the risk properties of their estimator.
Hence, apparently little is known about the behavior of the LSE $\hat{\theta}$ as an
estimator of $\theta^*$. 
The problem of unimodal density estimation is actually more well studied than its regression counterpart
\citep{birge1997estimation,eggermont2000maximum,bickel1996some,meyer2001alternative}. A recent work worth mentioning here is~\citet{balabdaoui2015maximum} where the authors study the estimation of discrete unimodal probability mass functions.

The isotonic regression problem is closely related,
except that the underlying sequence is assumed to be
nondecreasing instead of unimodal. As is clear, with the mode known,
fitting a least squares unimodal sequence reduces to fitting an
increasing sequence to the first part and a decreasing sequence to the
second part. The risk properties of the LSE in the isotonic regression
problem are fairly well understood, having been intensively studied
by a number of authors
\citep{vdG90,vdG93,Donoho91,BM93,Wang96,MW00,Zhang02,chatterjee2015risk}. 
In particular, \cite{Zhang02} shows the existence of a universal positive constant
$C$ such that
\begin{equation}\label{motw}
R(\theta^*, \hat{\theta}) = \frac{1}{n} \E \|\theta^* - \hat\theta\|^2
\leq C \left\{\left(\frac{\sigma^2
    V(\theta^*)}{n} \right)^{2/3} +  \frac{\sigma^2 \log n}{n} \right\}
\end{equation}
with $V(\theta^*) := \theta^*_n - \theta^*_1$. This result shows that
the risk of $\hat{\theta}$ scales as  $n^{-2/3}$ in the sequence
length,  provided $V(\theta^*)$ is bounded from above by a
constant. It can be proved that $n^{-2/3}$ is the minimax rate of
estimation in this problem, see~\citep{chatterjee2015risk}. 

For a monotonic vector $\theta \in \R^n$ we define $\s(\theta)$ to be the
cardinality of the set of values:
\begin{equation}
\s(\theta) = \bigl|\{\theta_1, \dots, \theta_n \}\bigl|. 
\end{equation}
Also let us define $\M_{n} = \{\theta \in \R^n:
\theta_1 \leq \dots \leq \theta_n\}$ to be the set of monotonic
sequences. A complementary upper bound on $R(\theta^*, \hat{\theta})$
in the isotonic regression problem has been proved by
\cite{chatterjee2015risk} who show that (for $n > 1$)
\begin{equation}\label{usa}
  R(\theta^*, \hat{\theta}) \leq 6 \inf_{\theta \in \M_n}
  \left(\frac{\|\theta^* - \theta\|^2}{n} + \frac{\sigma^2
    \s(\theta)}{n} \log \frac{en}{s(\theta)} \right) .
\end{equation}
This risk bound has recently been further improved
by~\cite{bellec:2016}, with the constant $6$ being improved to $1$,
and with a removal of the assumption that the true  underlying sequence $\theta^*$
is isotonic, yielding
\begin{equation}\label{bellecada}
  R(\theta^*, \hat{\theta}) \leq \inf_{\theta \in \M_n}
  \left(\frac{\|\theta^* - \theta\|^2}{n} + \frac{\sigma^2
    \s(\theta)}{n} \log \frac{en}{s(\theta)} \right) .
\end{equation}

The bounds \eqref{motw} and \eqref{bellecada} provide a nearly
complete understanding of the global accuracy of the LSE
$\hat{\theta}$ in isotonic sequence estimation. In particular, the risk
of the LSE can never be larger than the minimax rate $(\sigma^2
V(\theta^*)/n)^{2/3}$, while it can be the parametric rate (up to log factors)
$\sigma^2/n$, up to logarithmic multiplicative factors, if $\theta^*$
can be well-approximated by $\theta$ having small $\s(\theta)$. In
fact, the sharp oracle inequality~\eqref{bellecada} also implies parametric
rates for $\theta$ that is well approximated by a piecewise constant
sequence with not too many values. It also gives rates of convergence
for $\hat{\theta}$ to the projection of $\theta^*$ in the cone $\M_n$
in case $\theta^*$ is not in $\M_n$. So, in this sense, the LSE in the
isotonic regression problem is automatically adaptive to piecewise constant sequences. Such automatic adaptivity properties are also seen in other shape constrained
estimation problems such as convex regression
\citep{chatterjee2015risk,GSvex} and monotone matrix estimation
\citep{chatterjee2015matrix}.  The goal of the research leading to the
current paper was to extend the parameter space from isotonic to
unimodal sequences and investigate whether the risk
bounds~\eqref{motw} and~\eqref{bellecada} continue to hold. The
following section summarizes our findings.

\def\256{}

\section{Results}
Our first result establishes minimaxity of the least squares estimator
for unimodal sequence estimation.
\begin{theorem}\label{worstcase}
Fix any positive integer $n$ and $\theta^* \in \U_n$. Let $V(\theta^*)
= \max_{1 \leq i \leq n}\theta^*_{i} - \min_{1 \leq i \leq n}
\theta^*_{i}$. There exists a universal constant $C$ such that for any positive $\alpha > 0$ we have the following upper bound with probability not less than $1 - 2n^{- \alpha},$
\begin{equation*}
\frac{1}{n} \|\hat{\theta} - \theta^*\|^2 \leq
C \sigma^{4/3} \left(V(\theta^*) + \256 \sigma\right)^{2/3} n^{-2/3} +
(C + 24 \alpha) \sigma^2 \frac{\log n}{n}.
\end{equation*}
\end{theorem}

This is the analog of~\eqref{motw} in unimodal
regression. It shows that the least squares estimator $\hat{\theta}$ converges to
$\theta^*$ at the rate $n^{-2/3}$ in mean squared error loss,
uniformly for all unimodal sequences $\theta^*$ with $V(\theta^*)$ bounded. 
Since the minimax rate for the set of
isotonic sequences with $V(\theta^*) \leq C$ is 
$O(n^{-2/3})$, and unimodal regression is concerned with a larger parameter space than isotonic regression, this establishes that the least squares estimator is minimax rate optimal.

Our next theorem says that the least squares
estimator is adaptive to sequences that are piecewise constant, in the sense that the risk scales according to the number of pieces of the true sequence and has faster convergence rate than the worst case $O(n^{-2/3})$ rate.
\begin{theorem}\label{adaptive}
Fix a positive integer $n$ and $\theta^* \in \U_n$. Let $m^*$ be a
mode for $\theta^*,$ that is $\theta^* \in C_{m^*}$. Let $s_1$ equal
the number of distinct values of $(\theta^*_{1},\dots,\theta^*_{m^*})$
and $s_2$ equal the number of distinct values of $(\theta^*_{m^* +
1},\dots,\theta^*_{n})$. Fix any $\alpha > 0$. Then the mean squared error satisfies
\begin{equation}
\frac{1}{n} \|\hat{\theta} - \theta^*\|^2 \leq
12 \sigma^2 \left(\frac{s_1+s_2}{n}\right) \log\left(\frac{e n}{s_1+s_2}\right) +
48 (\alpha+2) \sigma^2 (s_1 + s_2) \frac{\log n}{n}
\end{equation}
with probability at least $1 - 4/n^\alpha$.
\end{theorem}

This theorem says that the mean squared error of the LSE scales for a unimodal sequence scales like $(s_1 + s_2) \log n/n$ where $s_1$ is the number of constant pieces of the part of $\theta^*$ which is nondecreasing and $s_2$ is the number of constant pieces of the part of $\theta^*$ which is nonincreasing. Therefore $s_1 + s_2$ could be thought of as roughly the number of steps (going up and coming down) in the sequence $\theta^*$.

As a consequence of the above theorem, when the true unimodal sequence has a bounded number of steps, the risk of the least squares estimator will decay
at the parametric rate of convergence $O(1/n)$ up to log factors.
As an illustration, if $\theta^*$ is the vector of evaluations of the
indicator function $f(x) =
\I\{0 \leq x \leq 1\}$ at $n$ (sorted) points on the real line, 
the risk of $\hat\theta$ would decrease at the parametric rate. 
The result shows that as long as $s_1 + s_2 = o(n^{1/3})$, 
rates of convergence that are faster than the global minimax rate are obtained.


\begin{remark}
Both Theorems~\ref{worstcase} and~\ref{adaptive} hold when the error
vector is Gaussian with independent entries, mean zero and variance
$\sigma^2$. They also hold when the error vector is composed of mean
zero independent entries with absolute value bounded by $\sigma >
0$. An important ingredient in our proofs is the concentration result
(Theorem~\ref{gaussconc}) for Lipschitz functions of a Gaussian random
vector. The proof for bounded errors follow by using Ledoux's
concentration result (Theorem~\ref{conc}) for convex Lipschitz
functions of the error vector.
\end{remark} 

\begin{remark}
The same upper bound of Theorem~\ref{adaptive} was obtained
independently by \cite{bellec:2016} and \cite{seriation}.
\end{remark}

\begin{remark}
All our risk bounds are shown to hold with high probability. These bounds can then be integrated to get bounds in expectation.
\end{remark}

\begin{remark}
The rate of convergence in Theorem~\ref{worstcase} scales with $n$ in
exactly the same way as in isotonic regression, in spite of the fact
that the parameter space is now a union of $n$ convex cones, each of
which is comprised of at most two isotonic pieces. In particular, the risk bound in
Theorem~\ref{worstcase} does not have a logarithmic factor of $n$.
\end{remark}



\begin{remark}
No smoothness conditions are assumed of the underlying unimodal
sequence. The least squares estimator is fully automated and does not
require any tuning parameters, as is often the case for
shape-constrained estimators.
\end{remark}

\begin{remark}
Fitting the unimodal least squares estimator can provide a means of
trading off computational time for statistical risk accuracy in a concave
regression problem.  To explain, note that the set of unimodal sequences contains
the set of concave sequences.  The latter has a minimax rate of $O(n^{-4/5})$, which is naturally faster than the rate in unimodal regression. However, to the best of our knowledge,
the running time for convex regression is at least quadratic, $O(n^2)$,
whereas runtime for unimodal regression is $O(n).$
While the unimodal estimator will not be concave in general, 
it will be close to the underlying concave sequence in an average
$\ell_2$ sense, for large $n$.
\end{remark}

The worst case rate of convergence (Theorem~\ref{worstcase})
of the LSE for unimodal regression matches that of 
isotonic regression, and the adaptive risk bound (Theorem~\ref{adaptive})
is almost as strong as the adaptive risk bound~\eqref{bellecada} for
isotonic regression. We say almost because Theorem~\ref{adaptive} is
useful for $\theta^*$ which is exactly piecewise constant with a few
pieces but may not be useful when $\theta^*$ is very well approximable
by a piecewise constant with a few pieces. Risk bounds which provide
such ``continuity'' of risk are often referred to as oracle inequalities. Previous proofs for oracle risk bounds in shape constrained problems appear to rely heavily on the
convexity of the parameter space; thus these techniques do
not readily apply to unimodal regression, where $\U_n$ is nonconvex.
In particular, existing proof techniques first bound the 
statistical dimension (see~\citet{amelunxen2014living}) for a closed convex cone when $\theta^*$ belongs to the 
lineality space of the cone (see~\citet{chatterjee2015risk}), and then derive risk bounds
for Gaussian widths of the tangent cone at $\theta^*$
(see Proposition 1 in~\cite{bellec:2016} and Lemma 4.1 in~\cite{chatterjee2015matrix}).
Such techniques rely on the convexity of
the parameter space and exploit the KKT conditions
for the projection onto closed convex sets. Our parameter
space is not convex and hence we are unable to directly use these
results to prove oracle risk bounds.

The main complication in this problem is to handle the unknown mode location. Our analysis reveals  that even with the unknown mode, it is still possible to derive adaptive risk bounds for piecewise
constant unimodal sequences, when the number of steps is not too
large. At a high level, the main idea here is to show first that the
risk depends on a local Gaussian width like term. This fact is now
known when the parameter space is convex and we are able to extend
this observation to our nonconvex parameter space. We analyze
this local Gaussian width like term and by various steps of
refinement we show that this term scales like the mean squared error
in isotonic regression plus logarithmic terms.

\section{LSE Slicing Lemma}

In this section we show that the loss
$\|\hat{\theta} - \theta^*\|$ has a
variational representation whenever the underlying vector
$\theta^*$ is known to belong to a parameter space $C$ that can
be expressed as a finite union of convex ``slices.'' 
The lemma is a deterministic identity that 
generalizes Proposition 1.3 in~\cite{Chat14}. For us, this lemma is directly applied in the proof of Theorem~\ref{adaptive}. This technique may be of use in other contexts; for example
the set of permutations arising in the analysis of pairwise
comparisons has this structure, see~\citet{shah2015stochastically},~\citet{chatterjee2016estimation}.

\begin{lemma}\label{chapro}
Let $C =
\cup_{m = 1}^{M} C_m$ where each $C_m \subset \R^n$ is a closed convex set.
Fix some $\theta^* \in C$, and let $y = \theta^* + z$.
Define the function $f_{\theta^*}: \R_{+} \rightarrow \R$ as 
\begin{equation}
f_{\bt^*}(t) = \sup_{\theta \in C: \|\theta - \theta^{*}\| \leq t} \langle z,\theta - \theta^{*}\rangle - \frac{t^2}{2}. 
\end{equation}
If $\hat{\theta} \in \argmin_{\theta \in C} \|y - \theta \|^2$
is a least squares estimator over $C$ then
\begin{equation}\label{deter}
\|\hat{\theta} - \theta^*\| \in \argmax_{t \geq 0} f_{\theta^*}(t).
\end{equation}
Moreover, if $t^*$ satisfies $f_{\theta^*}(t) < 0 $ for all $t \geq
t^*$, then
\begin{equation}\label{fact}
\|\hat{\theta} - \theta^{*}\| < t^*.
\end{equation} 
Relations \eqref{deter} and~\eqref{fact} are
deterministic, and do not depend on any distributional properties of the error vector $z$.
\end{lemma}

\begin{proof}
To prove~\eqref{deter}, first note that we can write 
\begin{align}
\hat{\theta} &\in \argmin_{\theta \in C} \left(\|z\|^2 + \|\theta -
\theta^*\|^2 - 2 \langle z, \theta - \theta^* \rangle\right) \\
&= \argmax_{\theta \in C} \left(\langle z, \theta - \theta^* \rangle - \|\theta - \theta^*\|^2/2\right).
\end{align}
It follows from the second equation that
\begin{align}
\|\hat{\theta} - \theta^*\| &\in \argmax_{t > 0} \sup_{\theta \in C:
  \|\theta - \theta^*\| = t} \langle z, \theta - \theta^* \rangle -
\frac{t^2}{2} \nonumber \\
&= \argmax_{t > 0} \left(\max_{1 \leq m \leq M } \:\:\sup_{\theta \in
  C_m: \|\theta - \theta^*\| = t} \langle z, \theta - \theta^* \rangle
- \frac{t^2}{2}\right) \nonumber \\
&= \argmax_{t > 0} \max_{1 \leq m \leq M} h_m(t) \label{deter1}
\end{align}
where we have defined the functions $h_{m}: \R_{+} \rightarrow \R$ as 
\begin{equation}
h_m(t) = \sup_{\theta \in C_m: \|\theta - \theta^*\| = t} \langle z, \theta - \theta^* \rangle - \frac{t^2}{2}.
\end{equation}
Now for each $1 \leq m \leq M$ define the functions $g_{m}: \R_{+} \rightarrow \R$ as 
\begin{equation}
g_{m}(t) = \sup_{\theta \in C_m: \|\theta - \theta^{*}\| \leq t} \langle z,\theta - \theta^{*}\rangle.
\end{equation}
If $d_m = \inf_{v \in \C_m} \|\theta^* - v\|$ is positive, then we
define $g_{m}(t) = - \infty$ whenever $t < d_m$.
Since the set $C_m$  is closed convex, it can be shown that $g_m$ is a
concave function of $t$. Such a calculation has been done in the proof of Theorem 1 in~\cite{Chat14} but for sake of completeness we prove the concavity of $g_m$ in Lemma~\ref{concav} in the appendix.
Next, define functions $f_{m}: \R_{+} \rightarrow \R$ according to
\begin{equation}
f_{m}(t) = g_{m}(t) - \frac{t^2}{2}.
\end{equation}
Then $f_m$ is strictly concave
as a function of $t$ whenever $t \geq d_m$. An application of the
Cauchy-Schwarz inequality shows that $f_{m}(t)$ decays to $-\infty$ as
$t \rightarrow \infty$. These facts establish that $f_{m}(t)$ has a
unique maximizer.

Let $t_m$ be the unique maximizer of $f_m$. We now show that $t_m$ is a
unique maximizer of $h_m$. It is clear from the definitions that for all
$t \geq 0$ we have $h_m(t) \leq f_m(t)$.  Recall that
\begin{equation}
f_m(t_m) =
\sup_{\theta \in C_m: \|\theta - \theta^*\| \leq t_m} \langle z,
\theta - \theta^* \rangle - \frac{t_m^2}{2},
\end{equation}
and let $\tilde{\theta} \in
\{\theta \in C_m: \|\theta - \theta^*\| \leq t_m\}$ be a point where
$f_{m}(t_m)$ is achieved. If it were the case that $\|\tilde{\theta} - \theta^*\| = t_0 <
t_m$, then we would have $f_m(t_0) > f_m(t_m)$, contradicting the
definition of $t_m$. Hence $\|\tilde{\theta} - \theta^*\| =
t_m$, implying that $f_m(t_m) = h_m(t_m)$. Therefore, 
\begin{equation}
h_m(x) \leq f_m(x) < f_m(t_m) = h_m(t_m)
\end{equation}
for any $x\neq t_m$.
This shows that $t_m$ is a unique maximizer of $h_m$ as well.
Therefore we have shown that
\begin{align}
\|\hat{\theta} - \theta^*\| &\in \argmax_{t > 0} \max_{1 \leq m \leq M}
h_m(t) \\
&= \argmax_{t > 0} \max_{1 \leq m \leq M} f_m(t) \\
& = \argmax_{t > 0} f_{\theta^*}(t),
\end{align}
thus proving \eqref{deter}. 

It remains to prove~\eqref{fact}. Since $\theta^* \in C$ there exists
$1 \leq m^* \leq M$ such that $\theta^* \in C_{m^*}$. Now it is easy
to see that $f_{m^*}(0) = 0$ and hence $\max_{t \geq 0}
f_{m^*}(t) \geq 0 $. 
Because of this fact and $f_{\theta^*}(t) = \max_{1 \leq m \leq M} f_{m}(t)$ we have
\begin{equation}
\max_{t \geq 0} f_{\theta^*}(t) \geq 0.
\end{equation}
This inequality combined with the representation
formula~\eqref{deter} proves~\eqref{fact}, finishing the proof
of the lemma.
\end{proof}

\section{Global Risk Bound}
The goal of this section is to prove Theorem~\ref{worstcase}. We prove both Theorem~\ref{worstcase} in this section and Theorem~\ref{adaptive} in the next section assuming Gaussian errors. The case of bounded errors will also follow from  our proofs. See Remark~\ref{bddnoise} for more details. Recall that for a
subset $\F \subseteq \R^{n}$ and $\epsilon > 0$, the
$\epsilon$-covering number $N(\epsilon, \F)$ is the minimum number of
balls of radius $\eps$ required to cover $\F$ in the Euclidean
norm. Throughout the following, $C$ will represent a universal
constant, whose particular value might change from calculation to
calculation. The following result follows from the metric entropy
results of monotone functions \citep{GW07}; a proof can be found in
Lemma 4.20 in \cite{Chat14}.  By symmetry, the same covering number
bound holds for nonincreasing sequences.
\begin{theorem}[Gao-Wellner]\label{monocover}
Let $\M_{[a,b]} = \{v \in \R^n: a \leq v_1 \leq \dots \leq v_n \leq
b\}$
denote the set of monotone sequences of length $n$ taking values between $a$ and $b$. Also recall that $\|\cdot\|$ denotes the standard Euclidean norm. Then for any $\eps > 0$,
\begin{equation}
\log N(\eps,\M_{[a,b]}, \|\cdot\|) \leq \frac{C \sqrt{n} (b - a)}{\eps}
\end{equation}
where $C$ is a universal constant.  
\end{theorem}

Throughout this section, $z\sim N(0,\sigma^2 I_n)$ will denote a Gaussian random variable.
We will require the following chaining bound (see e.g., \citet{VandegeerBook}).
\begin{theorem}[Chaining]
\label{dudthm}
For every $\bt^* \in \M$ and $t > 0$,
\begin{equation}\label{gaup}
\E \left( \sup_{\bt \in B(\bt^*, t)} \left<z, \bt - \bt^* \right> \right) \leq  12 \sigma \int_{0}^{t}
      \sqrt{\log N(\epsilon, B(\bt^*, t))} \;
     d\epsilon,
\end{equation}
where $B(\theta^*,t)$ denotes the standard Euclidean ball around
$\theta^*$ of radius $t$.
\end{theorem}

We will also require a standard Gaussian concentration inequality;
the proof of the following can be found in the argument after equation
(2.35) in~\cite{Ledoux01conc}.
\begin{theorem}[Gaussian concentration]
\label{gaussconc}
Let $f: \R^n \rightarrow \R$ be a function that is $L$-Lipschitz, so
that $|f(x) - f(y)| \leq L \|x - y\|$ for all $x$ and $y$. Then for any $t \geq 0$,
\begin{equation}
\P\left(f(z) \leq \E(f(z) + \sigma t\right) \leq \exp\left(-\frac{t^2}{2 L^2}\right).
\end{equation}
\end{theorem}

\begin{remark}\label{bddnoise}
Our entire analysis can be done when $z$ is a mean zero error vector of independent entries with absolute value bounded by a constant $\sigma > 0$. This is done by invoking the following concentration result whenever we have used the Gaussian concentration theorem for Lipschitz functions. We have used the Gaussian concentration theorem for random variables which can be seen as a supremum of linear functions of the error vector $z$ and hence are convex functions of $z$. This fact enables the application of Ledoux's concentration result wherever we have applied Theorem~\ref{gaussconc} .
\end{remark}
A proof of the following theorem can be found in~\citet{boucheron2013concentration}.
\begin{theorem}[Ledoux]{\label{conc}}
If $f:[-\sigma,\sigma]^n \rightarrow \R$ is a convex Lipschitz function with Lipschitz constant $L$ and $\eps$ is a mean zero random vector with independent entries in $[-\sigma,\sigma]$, then we have
\begin{equation}
\P(f(\eps) > u) \leq \exp(-u^2/8 \sigma^2 L^2). 
\end{equation}
\end{theorem}

Recall that $C_m = \{\theta \in \R^n: \theta_1 \leq \theta_2 \dots \leq
\theta_m \geq \theta_{m + 1} \dots \geq \theta_n\}$
is the collection of unimodal sequences with a mode at $m$. 
We now prove a key lemma controlling the term $\E
\big(\sup_{\theta \in C_m: \|\theta - \theta^*\| \leq t} \langle z,
\theta - \theta^* \rangle\big)$ as a function of $t$ whenever the
underlying $\theta^*$ is monotonic.
This will be crucial in finally proving Theorem~\ref{worstcase}.

The idea of the proof is as follows. The standard technique in
empirical process theory of upper bounding the expected Gaussian term
$\E \big(\sup_{\theta \in C_m: \|\theta - \theta^*\| \leq t} \langle
z, \theta - \theta^* \rangle\big)$ is to use the Dudley's entropy
integral bound, as given in Theorem~\ref{dudthm}. This requires tight
upper bounds on the covering number for the set $\{\theta \in C_m:
\|\theta - \theta^*\| \leq t\}$. Now, upper bounds can be derived
for the log-covering number of bounded sequences in $C_m$ using
Theorem~\ref{monocover}. However, a direct application of Theorem~\ref{monocover} to cover the set
$\{\theta \in C_m: \|\theta - \theta^*\| \leq t\}$ would give us loose
bounds, as $t$ can grow with $n$. Instead, we carry out an extra
peeling step. In particular, for any $\theta$ belonging to $\{\theta \in C_m: \|\theta
- \theta^*\| \leq t\}$, we define a truncated version $\theta^{'}$
belonging to the same set that is also bounded by a constant factor
$L = C \big(V(\theta^*) + \sigma\big)$.
Note that one can write
\begin{equation}\label{brk}
\E \Bigl(\sup_{\theta \in C_m: \|\theta - \theta^*\| \leq t} \langle z,
\theta - \theta^* \rangle\Bigr) \leq \E \Bigl(\sup_{\theta \in C_m: \|\theta - \theta^*\| \leq t} \langle z, \theta - \theta^{'} \rangle\Bigr) + \E \Bigl(\sup_{\theta^{'} \in C_m: \|\theta^{'} - \theta^*\| \leq t, \max_{i} |\theta^{'}_{i}| \leq L} \langle z, \theta^{'} - \theta^{*} \rangle\Bigr).
\end{equation}
The truncation is defined in such a way that the first term on the 
right side of the above inequality is small enough for our
purposes. The second term can be controlled by a direct
application of Dudley's entropy integral inequality. This extra
peeling step helps us to derive a tight risk bound.

The definition of these truncations forms a key part of our argument. We actually form 
a truncation $\theta'$ of an arbitrary unimodal sequence $\theta\in C_m$
with respect to a fixed monotone sequence $\theta^*$. The
tails of $\theta$ are raised to $\theta_1^*-L$
over $S_1^L\cup S_1^R = S_1 = \{i : \theta_i < \theta_1^* - L\}$.
In the interval $S_2 = \{i : \theta_i > \theta_m^* + L\}$ around the
mode $m$, the sequence is lowered to the level
$\theta_n^* +L$. It is crucial that we are able to choose $L$ to be a
constant factor while still maintaining the fact that the first term on
the right side of~\eqref{brk} is sufficiently small. Figure~1
shows a schematic of the construction of our truncation $\theta'$ of $\theta$; the figure
may be helpful in understanding the steps of the proof.

\begin{lemma}\label{key1}
Fix  any nondecreasing sequence $\theta^* \in \M_n$ and let $1\leq m\leq n$. Then for all $t \geq 0$,
\begin{equation}
\E \left(\sup_{\theta \in C_m: \|\theta - \theta^*\| \leq t} \langle z, \theta - \theta^* \rangle\right) \leq C \sigma \left(n^{1/4} t^{1/2} \sqrt{\big(V(\theta^*) + \sigma\big)}\right) + \frac{t^2}{8}
\end{equation}
where $V(\theta^*) = \max_{1 \leq i \leq n} \theta^*_i - \min_{1 \leq i \leq n} \theta^*_i$.
The bound also holds, by symmetry, for any nonincreasing sequence
$\theta^* \in \M^{-}_n$.
\end{lemma}


\begin{proof}
Let $K = \{\theta \in C_m: \|\theta - \theta^*\| \leq t\}$, and
define $K^{'} \subset K$ as 
\begin{equation}
K' = \left\{\theta \in K: \max_{1 \leq i \leq n} \theta_i \leq \theta^*_{n}
+ L, \min_{1 \leq i \leq n} \theta_i \geq \theta^*_{1} - L\right\}
\end{equation} where $L$ is a fixed positive number to be chosen later.
For any $\theta \in K$ we will define a truncated version of $\theta$ belonging to $K'$ which will be denoted by $\theta'$. Then we will have the inequality
\begin{equation}\label{eqbrk}
\E \sup_{\theta \in K} \langle z, \theta - \theta^{*} \rangle \leq \E \sup_{\theta \in K} \langle z, \theta - \theta' \rangle + \E \sup_{\theta' \in K'} \langle z, \theta' - \theta^{*} \rangle.
\end{equation}
Fix an arbitrary $\theta \in K$. Consider the sets 
$S_1 = \{i: \theta_i < \theta^*_{1} - L\}$ and $S_2 = \{i: \theta_i > \theta^*_{n} 
+ L\}$ 
and define $\theta'$ according to
\begin{equation}
\theta'_i = \begin{cases}
  \theta_1^* - L & \text{if $i\in S_1$} \\
  \theta_n^* + L & \text{if $i\in S_2$} \\
  \theta_i & \text{otherwise}.
\end{cases}
\end{equation}

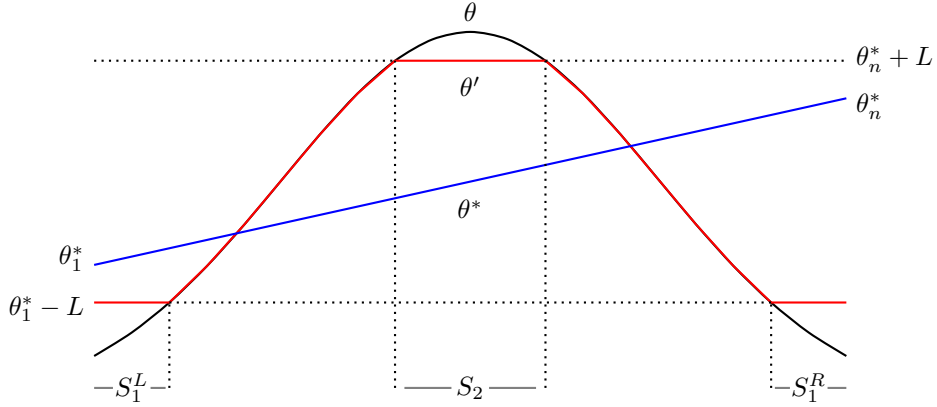
\begin{figure}
\label{fig:lemma4p2}
\begin{center}
\begin{tikzpicture}[thick, scale=2.5]
\draw [black] plot[smooth] coordinates {(-2.000000, 0.269955) (-1.800000, 0.394751) (-1.600000, 0.554604) (-1.400000, 0.748637) (-1.200000, 0.970930) (-1.000000, 1.209854) (-0.800000, 1.448458) (-0.600000, 1.666123) (-0.400000, 1.841351) (-0.200000, 1.955213) (0.000000, 1.994711) (0.200000, 1.955213) (0.400000, 1.841351) (0.600000, 1.666123) (0.800000, 1.448458) (1.000000, 1.209854) (1.200000, 0.970930) (1.400000, 0.748637) (1.600000, 0.554604) (1.800000, 0.394751) (2.000000, 0.269955) };

\draw [red] plot coordinates {(-2.000000, 0.554604) (-1.800000, 0.554604) (-1.600000, 0.554604) (-1.400000, 0.748637) (-1.200000, 0.970930) (-1.000000, 1.209854) (-0.800000, 1.448458) (-0.600000, 1.666123) (-0.400000, 1.841351) (-0.200000, 1.841351) (0.000000, 1.841351) (0.200000, 1.841351) (0.400000, 1.841351) (0.600000, 1.666123) (0.800000, 1.448458) (1.000000, 1.209854) (1.200000, 0.970930) (1.400000, 0.748637) (1.600000, 0.554604) (1.800000, 0.554604) (2.000000, 0.554604) };

\draw [dotted] plot coordinates {(-1.600000, 0.554604) (1.600000, 0.554604) };

\draw [dotted] plot coordinates {(-2.000000, 1.841351) (-0.400000, 1.841351) };

\draw [dotted] plot coordinates {(0.400000, 1.841351) (2.000000, 1.841351) };

\draw [blue] plot coordinates {(-2.000000, 0.754604) (2.000000, 1.641351) };

\node[above] at (0,2.00) {$\theta$};

\node[above] at (0,1.6) {$\theta'$};

\node[above] at (0,0.95) {$\theta^*$};

\node[right] at (2, 1.85) {$\theta_n^*+L$};

\node[right] at (2, 1.60) {$\theta_n^*$};

\node[left] at (-2, 0.79) {$\theta_1^*$};

\node[left] at (-2, 0.52) {$\theta_1^*-L$};

\draw [dotted] plot coordinates {(0.400000, 0.100000) (0.400000, 1.841351) };

\draw [dotted] plot coordinates {(-0.400000, 0.100000) (-0.400000, 1.841351) };

\draw [dotted] plot coordinates {(-1.600000, 0.100000) (-1.600000, 0.554604) };

\draw [dotted] plot coordinates {(1.600000, 0.100000) (1.600000, 0.554604) };

\draw [gray] (-.35,.1) -- (-.1, .1);
\draw [gray] (.1,.1) -- (.35, .1);
\node at (0,.1) {\footnotesize $S_2$};

\draw [gray] (-1.7,.1) -- (-1.62, .1);
\draw [gray] (-2.0,.1) -- (-1.9, .1);
\node at (-1.8,.1) {\footnotesize $S_1^L$};

\draw [gray] (1.7,.1) -- (1.62, .1);
\draw [gray] (2.0,.1) -- (1.9, .1);
\node at (1.8,.1) {\footnotesize $S_1^R$};

\end{tikzpicture}
\caption{The proof of Lemma~\ref{key1} is based on
a truncation $\theta'$ of an arbitrary unimodal sequence $\theta\in C_m$
with respect to a fixed monotone sequence $\theta^*$. The
tails of $\theta$ are raised to $\theta_1^*-L$
over $S_1^L\cup S_1^R = S_1 = \{i : \theta_i < \theta_1^* - L\}$.
In the interval $S_2 = \{i : \theta_i > \theta_n^* + L\}$ around the
mode $m$, the sequence is lowered to the level
$\theta_n^* +L$.}
\end{center}
\end{figure}

As indicated in Figure~1, $S_1$ is the union of 
a left prefix $S_1^L$ of $\{1,\ldots, n\}$ and a right suffix $S_1^R$.
It is then clear that 
$\min_{1 \leq i \leq n} \theta'_{1} \geq \theta^{*}_{1} - L$ and $\max_{1 \leq i \leq n} \theta'_{i} \leq \theta^{*}_{n} + L$
Also,  by construction of $\theta'$ it is unimodal, and we have the  contractive property 
\begin{equation}\label{contrac}
|\theta^*_i - \theta'_i| \leq |\theta^*_i - \theta_i|
\end{equation}
for any $1 \leq i \leq n$.
These facts show that $\theta' \in K'$. 

We now proceed to control the first term on the right side of the inequality in~\eqref{eqbrk}.
Using the definition of $\theta'$, we have
\begin{align}\label{eqtrunc}
\sum_{i = 1}^{n} z_i (\theta_i - \theta'_i) 
& = \sum_{i\in S_1} z_i (\theta_i - \theta'_i) + \sum_{i\in S_2} z_i (\theta_i - \theta'_i)\\
& \leq \sum_{i\in S_1} |z_i| (\theta_i' - \theta_i) + \sum_{i\in S_2} |z_i| (\theta_i - \theta'_i)\\
& = \sum_{i \in S_1} \sum_{j = 0}^{\infty} |z_i| (\theta_i' - \theta_i) \:\:\I\{2^j L < \theta^{*}_1 - \theta_i \leq 2^{j + 1} L\}
\\
&\qquad +  \sum_{i \in S_2} \sum_{j = 0}^{\infty} |z_i| (\theta_i -
\theta'_i) \:\:\I\{2^j L < \theta_i - \theta^*_{n} \leq 2^{j + 1} L\}
 \\
& \leq \sum_{j = 0}^{\infty} 2^{j + 1} L \sum_{i \in S_1}
 |z_i|\:\:\I\{2^j L < \theta^{*}_1 - \theta_i \leq 2^{j + 1} L\} 
\\
& \qquad 
+  \sum_{j = 0}^{\infty} 2^{j + 1} L \sum_{i \in S_2} |z_i|\:\:\I\{2^j L < \theta_i - \theta^*_{n} \leq 2^{j + 1} L\}
\end{align}
where $\I$ denotes the indicator function and the last inequality follows from the inequalities
\begin{align}
\theta_i' - \theta_i \leq \theta_1^* - \theta_i & \;\; \text{for $i\in S_1$}\\[5pt]
\theta_i - \theta_i' \leq \theta_i - \theta_n^* & \;\; \text{for $i\in S_2$}.
\end{align}
We now note that for any $\theta \in K$, since $\|\theta - \theta^{*}\| \leq
t$,
\begin{align}\label{l2con}
\left| \bigl\{ i : 2^j L < |\theta_i - \theta_i^*| \bigr\}\right| \leq \frac{t^2}{2^{2j} L^2} \equiv v_j.
\end{align}
Also note that if $i\in S_1$ then $\theta_1^* - \theta_i > 2^j L$
implies that $\theta_i^* - \theta_i > 2^j L$, since $\theta^*$ is
monotonic. Therefore,
\begin{equation}
\left|\bigl\{ i \in S_1 : 2^j L < \theta^*_1 - \theta_i\bigr\}\right| \leq v_j.
\end{equation}

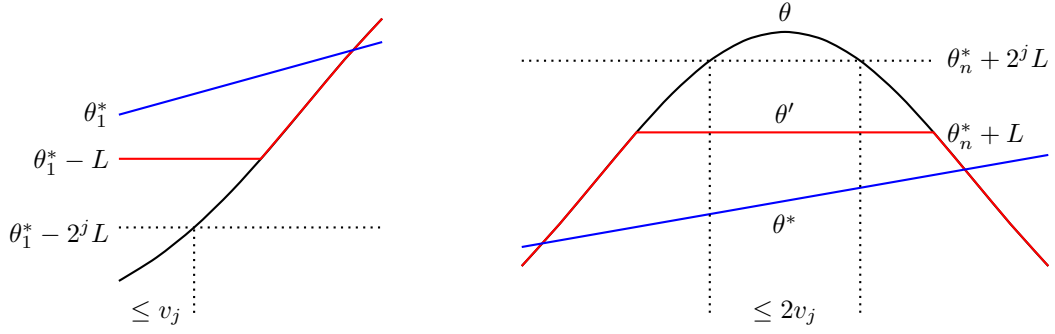
\begin{figure}
\begin{center}
\begin{tabular}{cc}
\qquad\qquad 
\begin{tikzpicture}[thick, scale=2.5]
\draw [black] plot[smooth] coordinates {(-2.000000, 0.269955) (-1.800000, 0.394751) (-1.600000, 0.554604) (-1.400000, 0.748637) (-1.200000, 0.970930) (-1.000000, 1.209854) (-0.800000, 1.448458) (-0.600000, 1.666123)  };
\draw [red] plot coordinates {(-2.000000, 0.92) (-1.800000, 0.92) (-1.600000, 0.92) (-1.400000, 0.92) (-1.2450000, 0.92) (-1.000000, 1.209854) (-0.800000, 1.448458) (-0.600000, 1.666123)  };
\draw [dotted] plot coordinates {(-2.000000, 0.554604) (-0.600000, 0.554604) };
\draw [blue] plot coordinates {(-2.000000, 1.154604) (-.6000000, 1.541351) };
\node[left] at (-2, 1.155) {$\theta_1^*$};
\node[left] at (-2, 0.90) {$\theta_1^*- L$};
\node[left] at (-2, 0.52) {$\theta_1^*- 2^jL$};
\draw [dotted] plot coordinates {(-1.600000, 0.100000) (-1.600000, 0.554604) };
\node at (-1.8,.1) {\footnotesize $\leq v_j$};
\end{tikzpicture}
&
\qquad\qquad 
\begin{tikzpicture}[thick, scale=2.5]
\draw [black] plot[smooth] coordinates {(-1.400000, 0.748637) (-1.200000, 0.970930) (-1.000000, 1.209854) (-0.800000, 1.448458) (-0.600000, 1.666123) (-0.400000, 1.841351) (-0.200000, 1.955213) (0.000000, 1.994711) (0.200000, 1.955213) (0.400000, 1.841351) (0.600000, 1.666123) (0.800000, 1.448458) (1.000000, 1.209854) (1.200000, 0.970930) (1.400000, 0.748637) };
\draw [red] plot coordinates {(-1.400000, 0.748637) (-1.200000, 0.970930) (-1.000000, 1.209854) (-0.79, 1.46) (-0.600000, 1.46) (-0.400000, 1.46) (-0.200000, 1.46) (0.000000, 1.46) (0.200000, 1.46) (0.400000, 1.46) (0.600000, 1.46) (0.790000, 1.46) (1.000000, 1.209854) (1.200000, 0.970930) (1.400000, 0.748637) };
\draw [dotted] plot coordinates {(-1.400000, 1.841351) (0.79, 1.841351) };
\draw [blue] plot coordinates {(-1.400000, 0.85) (1.400000, 1.34) };
\node[above] at (0,2.00) {$\theta$};
\node[above] at (0,1.46) {$\theta'$};
\node[above] at (0,0.87) {$\theta^*$};
\node[right] at (0.81, 1.85) {$\theta_n^*+2^jL$};
\node[right] at (0.81, 1.44) {$\theta_n^*+L$};
\draw [dotted] plot coordinates {(0.400000, 0.50) (0.400000, 1.841351) };
\draw [dotted] plot coordinates {(-0.400000, 0.50) (-0.400000, 1.841351) };
\node at (0,.5) {\footnotesize $\leq 2v_j$};
\end{tikzpicture}
\end{tabular}
\caption{The 
set $\bigl\{i \in S_1: 2^j L < \theta^*_1 - \theta_i \bigr\}$ is
the union of at most two intervals. Each has size no larger than
$v_j$. The figure on the left indicates the left interval.
Similarly, by unimodality, the set
$\bigl\{ i \in S_2 : 2^j L < \theta_i - \theta^*_n\bigr\}$
if it is nonempty, is an interval of length no greater than
$2v_j$.}
\end{center}
\end{figure}

Now observe that since $\theta$ is unimodal, any set of the form $\{i:
\theta_i < a\}$ for some number $a$ is necessarily a union of at most
two intervals; see Figure~2.  Therefore,
\begin{align}
\bigl\{i \in S_1: 2^j L < \theta^{*}_1 - \theta_i \leq 2^{j + 1} L\bigr\} 
& \subseteq \bigl\{i \in S_1: 2^j L < \theta^*_1 - \theta_i \bigr\} \\
& \subseteq  \{1,\ldots, v_j\} \cup \{n-v_j+1, n-v_j+2, \ldots, n\},
\end{align}
since each interval must have size no greater than $v_j$.
Similarly, we have that 
\begin{equation}
\left|\bigl\{ i \in S_2 : 2^j L < \theta_i - \theta^*_n\bigr\}\right| \leq v_j.
\end{equation}
Since $\theta \in C_m$ is unimodal, 
any set of the form $\{i: \theta_i > a\}$ 
for some number $a$ is necessarily an interval containing $m$, if it is
nonempty. Therefore, we have that
\begin{align}
\bigl\{i \in S_2: 2^j L < \theta_i - \theta^{*}_n \leq 2^{j + 1} L\bigr\}
&\subseteq \bigl\{i \in S_2: 2^j L< \theta_i - \theta^{*}_n \bigr\} \\
&\subseteq \bigl\{m - v_j + 1, \ldots, m + v_j - 1\bigr\}.
\end{align}
We conclude that
\begin{equation}
\sum_{i = 1}^{n} (\theta_i - \theta'_i) z_i \leq \sum_{j = 0}^{\infty} 2^{j + 1} L \left(\sum_{i = 1}^{v_j} |z_i| + \sum_{i = m - v_j + 1}^{m + v_j - 1} |z_i| + \sum_{i = n - v_j + 1}^{n} |z_i|\right).
\end{equation}
Note that this upper bound does not depend on
$\theta$, and 
that $\theta$ was an arbitrary element of $K$. Thus, using the fact $\E |z_i| = \sigma \sqrt{2/\pi}$
and~\eqref{l2con}, we arrive at 
\begin{align}
\E\left( \sup_{\theta \in K} \sum_{i = 1}^{n} (\theta_i - \theta'_i)
z_i\right) & \leq 4 L \sigma \sqrt{2/\pi} \sum_{j = 0}^{\infty} 2^{j + 1}  v_j 
 = \frac{ 4 t^2 \sigma \sqrt{2/\pi}}{L} \sum_{j = 0}^{\infty} 2^{1 - j} 
=  \frac{16 t^2 \sigma \sqrt{2/\pi}}{L}.
\end{align}
We now set $L = 128 \sigma \sqrt{2/\pi}$ to finally obtain
\begin{equation}\label{app2}
\E \left(\sup_{\theta \in K} \langle z,\theta - \theta' \rangle \right)\leq \frac{t^2}{8}.
\end{equation}



To control the second term on the right side of~\eqref{eqbrk} we set
$\A = K'$ and $\delta = 0$ in the chaining result~\eqref{dudthm} to obtain
\begin{equation}\label{dudcal}
\E \left(\sup_{\theta' \in K'} \langle z, \theta' - \theta^{*} \rangle\right)\leq 12 \sigma \int_{0}^{t} \sqrt{\log N(\eps,K'})\;d\eps .  
\end{equation}
By definition of $K'$ we can now apply Theorem~\ref{monocover}
with $a = \theta^*_{1} - L$ and $b = \theta^*_{n} + L$ to obtain 
\begin{equation}
\log N(\eps,K') \leq C \sqrt{n} \frac{\big(V(\theta^*) + 256
  \sigma \sqrt{2/\pi} \big)}{\eps}.
\end{equation}
Using~\eqref{dudcal} and integrating the above expression gives us,
for an appropriate constant $C$, 
\begin{equation}\label{eststep}
\E \left(\sup_{\theta' \in K'} \langle z, \theta' - \theta^{*} \rangle
\right)\leq C \sigma n^{1/4} t^{1/2} \sqrt{\big(V(\theta^*) + \sigma\big)}.
\end{equation}
Combining the last equation with~\eqref{app2} and~\eqref{eqbrk} finishes the proof of the lemma.
\end{proof}

We now prove a lemma that controls the expected Gaussian supremum where the supremum is taken over all
unimodal sequences.
\begin{lemma}\label{key2}
Fix a positive integer $n$ and a nondecreasing or nonincreasing sequence $\theta^*$. For all $t \geq 0$,
\begin{equation}
\E \left(\sup_{\theta \in \U_n: \|\theta - \theta^*\| \leq t} \langle z, \theta - \theta^* \rangle\right) \leq C \sigma \left(n^{1/4} t^{1/2} \sqrt{\big(V(\theta^*) + \256 \sigma\big)} + t \sqrt{\log n}\right) + \frac{t^2}{8}.
\end{equation}
\end{lemma}

\begin{proof}
We prove the lemma when $\theta^* \in \M_n$ is a nondecreasing
sequence; the proof when $\theta^*$ is nonincreasing is analogous. For each $1 \leq m \leq n$ and $t > 0$ define the random variables
\begin{equation}\label{def1}
X_m(t) = \sup_{\theta \in \C_m: \|\theta - \theta^*\| \leq t} \langle z, \theta - \theta^* \rangle.
\end{equation}
We first note that 
\begin{equation}\label{def2}
\sup_{\theta \in \U_n: \|\theta - \theta^*\| \leq t} \langle z, \theta - \theta^* \rangle = \max_{1 \leq m \leq n} X_m(t).
\end{equation}
Applying Lemma~\ref{gaulip} we see that the random variables $X_m(t)$
are Lipschitz functions of $z$ with Lipschitz constant $t$. Hence,
using the Gaussian concentration result given in Theorem~\ref{gaussconc} 
we get for all $x > 0$ and all $1 \leq m \leq n$,
\begin{equation}
\P\left(X_m(t) \leq \E X_m(t) + t \sigma x \right) \leq \exp\left(-\frac{x^2}{2}\right).
\end{equation}
A standard argument involving maxima of random variables with
sub-Gaussian tails is given in Lemma~\ref{subg}. Using this lemma and the
last equation we get that, for a universal constant $C$,
\begin{equation}
\E \max_{1 \leq m \leq n} X_m(t) \leq \max_{1 \leq m \leq n} \E X_m(t) + C \sigma t \sqrt{\log n}.
\end{equation}
Applying Lemma~\ref{key1} to the term $\max_{m} \E
X_m(t)$ completes the proof of the lemma.
\end{proof}

We are now ready to prove Theorem~\ref{worstcase}.
\begin{proof}[Proof of Theorem~\ref{worstcase}]
We start with the basic inequality $\|\hat{\theta} - \theta^*\|^2 \leq 2 \langle z, \hat{\theta} - \theta^*\rangle$ and rewrite it as
\begin{equation}\label{basic1}
\frac{1}{2} \|\hat{\theta} - \theta^*\| \leq \frac{1}{\|\hat{\theta} - \theta^*\|} \langle z, \hat{\theta} - \theta^*\rangle.
\end{equation}
This pointwise inequality follows from the fact that $\hat{\theta},$
among all $\theta \in \U$, maximizes the expression $$ g(\theta) =
\langle z, \theta - \theta^* \rangle - \|\theta - \theta^*\|^2/2.$$
Since $\theta^* \in \U$, we then have $g(\hat{\theta}) \geq
g(\theta^*) = 0$ which is equivalent to~\eqref{basic1}. Now, an
application of the Cauchy-Schwarz inequality to~\eqref{basic1} implies 
\begin{equation}\label{trivialbd}
\frac{1}{2} \|\hat{\theta} - \theta^*\| \leq \|z\|.
\end{equation}
Define the event $\A_1 = \{\|z\|^2 < 5n\}$. By a standard tail inequality for chi squared random variables \citep{laurent2000adaptive}, one can show that 
\begin{equation}\label{chisqp}
P(\A_1) \geq 1 - \exp(-n).
\end{equation}

The following argument conditions on the event $\A_1$ throughout.
For any $t > 0$ we can write
\begin{align}\label{basic2}
\frac{1}{2} \|\hat{\theta} - \theta^*\| & \leq \frac{t}{2} + \frac{1}{\|\hat{\theta} - \theta^*\|} \langle z, \hat{\theta} - \theta^*\rangle \I\{t < \|\hat{\theta} - \theta^*\| < \sqrt{5n}\}\\
\label{basic3}
& \leq \frac{t}{2} + \underbrace{\sup_{\theta \in \U: t < \|\theta - \theta^*\| \leq \sqrt{5n}} \frac{\langle z, \theta - \theta^*\rangle}{\|\theta - \theta^*\|}}_{H_t}. 
\end{align}
We will now control the random variable $H_t$ defined above for any $t \geq 1$. Writing $H_t$ as 
\begin{equation*}
H_t \leq \sup_{0 \leq l < \log_2(\sqrt{5n})} \sup_{\theta \in \U:2^l t < \|\theta - \theta^*\| \leq 2^{l + 1}t} \frac{\langle z, \theta - \theta^*\rangle}{\|\theta - \theta^*\|},
\end{equation*}
we now claim that instead of dividing the inner product term by the
variable factor $\|\theta - \theta^*\|$,
we can instead divide by the constant factor $2^{l} t$. An
application of Lemma~\ref{supkhela}, stated and proved in the
appendix, makes this precise and we can thus write
\begin{align}
H_t &\leq \sup_{0 \leq l < \log_2(\sqrt{5n})} \,\; \sup_{\theta \in \U:2^l
  t < \|\theta - \theta^*\| \leq 2^{l + 1}t} \frac{\langle z, \theta -
  \theta^*\rangle}{2^{l} t} \\
& \leq \frac{2}{t} \sup_{0 \leq l < \log_2(\sqrt{5n})} \,\; \sup_{\theta \in \U: \|\theta - \theta^*\| \leq 2^{l + 1} t} \frac{\langle z, \theta - \theta^*\rangle}{2^{l + 1}}\\
& = \frac{2}{t} \sup_{0 \leq l < \log_2(\sqrt{5n})} \,\; \sup_{w_l \in \U: \|w_l - w^*_l\| \leq t} \langle z, w_l - w^*_l \rangle
\end{align}
where $w^*_l = \theta^*/(2^{l + 1})$.
The random variables $W_l = \sup_{w_l \in \U: \|w_l - w^*_l\| \leq t} \langle z, w_l - w^*_l \rangle$ are Lipschitz functions of $z$ with Lipschitz constant $t$ as can be seen by applying Lemma~\ref{gaulip}. Therefore for each fixed integer $0 \leq l \leq \log_2(\sqrt{5n})$ the Gaussian concentration result in Theorem~\ref{gaussconc} shows for any $x > 0$,
\begin{equation*}
\P(W_l \leq \E W_l + t \sigma x) \geq 1 - \exp(-x^2/2).
\end{equation*}
Now define the event 
$$\A_2 = \Bigl\{\max_{0 \leq l \leq \log_2(\sqrt{5n})} W_l \leq
\max_{0 \leq l \leq \log_2(\sqrt{5n})} \E W_l + t \sigma x\Bigr\}.
$$
A union bound over $0 \leq l \leq \log_2(\sqrt{5n})$ then gives us
\begin{equation*}
\P(\A_2) \geq 1 - (1+ \log_2(\sqrt{5n})) \exp(-x^2/2).
\end{equation*}
Now, for any $0 \leq l \leq \log_2(\sqrt{5n})$ we can break up $\sup_{w_l \in \U: \|w_l - w^*_l\| \leq t} \langle z, w_l - w^*_l \rangle$ into two parts; the first part where $w^*$ is increasing and then 
the second part where $w^*$ is decreasing. We apply Lemma~\ref{key2} to both the parts to obtain
\begin{equation*}
\E W_l \leq C \sigma \left(n^{1/4} t^{1/2} \sqrt{\big(V(w^*_l) + \256 \sigma\big)} + t \sqrt{\log n}\right) + \frac{t^2}{8}.
\end{equation*}
Note that since $w^*_l = \theta^*/2^{l + 1}$ we have $V(w^*_l) \leq V(\theta^*)$ and hence we can write for all $t \geq 0,$
\begin{equation*}
\max_{0 \leq l \leq \log_2(\sqrt{5n})} \E W_l \leq C \sigma \left(n^{1/4} t^{1/2} \sqrt{\big(V(\theta^*) + \256 \sigma\big)} + t \sqrt{\log n}\right) + \frac{t^2}{8}.
\end{equation*}
We thus conclude that conditioned on the event $\A_2$, for all $t \geq 0$,
\begin{align}
H_t & \leq \frac{2}{t} \left(C \sigma \left(n^{1/4} t^{1/2} \sqrt{\left(V(\theta^*) + \256 \sigma\right)} + t \sqrt{\log n}\right) + \frac{t^2}{8} + t \sigma x\right)\\
 & = C \sigma \left(n^{1/4} t^{-1/2} \sqrt{\left(V(\theta^*) + \256 \sigma\right)} +  \sqrt{\log n}\right) + \frac{t}{4} + 2 \sigma x.
\end{align}
Together with \eqref{basic3} we then obtain the upper bound 
\begin{equation*}
\frac{1}{2} \|\hat{\theta} - \theta^*\| \leq \frac{t}{2} + C \sigma \left(n^{1/4} t^{-1/2} \sqrt{\left(V(\theta^*) + \256 \sigma\right)} +  \sqrt{\log n}\right) + \frac{t}{4} + 2 \sigma x
\end{equation*}
for all $t \geq 0$.
Optimizing over $t$ by setting
$$t = \max\{C \sigma^{2/3} \left(V(\theta^*) + \256
\sigma\right)^{1/3} n^{1/6},1\}$$ 
then gives us the desired bound on
$\|\hat{\theta} - \theta^*\|$. This bound holds on the set $\A_1 \cap
\A_2$. For any fixed $\alpha > 0$, set 
\begin{equation*}
x = \sqrt{2 \log(1+\log_2(\sqrt{5n})) + 2\alpha \log n}
\end{equation*} 
to conclude that
$P(\A_2) \geq 1 - n^{- \alpha}$. Also by~\eqref{chisqp} we then can
conclude $P(\A_1 \cap \A_2) \geq 1 - n^{- \alpha} - \exp(-n) \geq 1 - 2n^{-\alpha}$. Finally,
one can use the standard inequality $(a + b + c)^2 \leq 3 (a^2 + b^2 +
c^2)$ for nonnegative numbers $a,b,c$ to convert the bound on
$\|\hat{\theta} - \theta^*\|$ into a bound on $\|\hat{\theta} -
\theta^*\|^2$, completing the proof of the theorem.
\end{proof}

\section{Adaptive Risk Bound}
The goal of this section is to prove Theorem~\ref{adaptive}. The
main technical tool is Lemma~\ref{chapro}, used to show
that the local Gaussian width term $\sup_{\theta \in \U_n:\|\theta -
  \theta^*\| \leq t} \langle z, \theta - \theta^* \rangle$
scales only logarithmically with $n$ when the true $\theta^*$ is piecewise
constant with few pieces. This is the content of our next lemma when
$\theta^*$ is monotone increasing or decreasing.  

\begin{lemma}\label{lemmaadap}
Let $\theta^* \in \R^n$ be a monotone nondecreasing sequence in $\M_n$
with a constant number of pieces $s$. Fix any $\alpha > 0$. Then the following upper bound holds simultaneously for all $t \geq 0$ with probability not less than $1 - 2 n^{-\alpha}$:
\begin{equation*}
\sup_{\theta \in \U_n:\|\theta - \theta^*\| \leq t} \langle z, \theta - \theta^* \rangle  \leq 2 t \sigma \sqrt{s \log\Bigl(\frac{en}{s}\Bigr)} +  2 t \sigma \sqrt{2 s (\alpha + 2) \log n}.
\end{equation*}
\end{lemma}

\begin{proof}
Note that $\sup_{\theta \in \U_n:\|\theta - \theta^*\| \leq t} \langle z, \theta - \theta^* \rangle  = \max_{1 \leq m \leq n} X_{m}(t)$ where $X_{m}(t)$ is defined as
\begin{equation*}
X_m(t) = \sup_{\theta \in C_m:\|\theta - \theta^*\| \leq t} \langle z, \theta - \theta^* \rangle. 
\end{equation*}
First we control the term $X_{m}(t)$ for each $1 \leq m \leq
n$. Fixing a mode location $m$ we have
\begin{equation}\label{firstsecond}
X_{m}(t) \leq \sup_{\theta \in \M_{m}: \|\theta - \theta^*_{1:m}\| \leq t} \langle z_{1:m}, \theta - \theta^*_{1:m}\rangle + \sup_{-\theta \in \M_{n - m}: \|\theta - \theta^*_{(m + 1):n}\| \leq t} \langle z_{(m + 1):n}, \theta - \theta^*_{(m + 1):n}\rangle 
\end{equation}
where it should be understood that the first term on the right side of
the above inequality involves the first $m$ coordinates of the
relevant vectors and the last term involves the last $n - m$
coordinates of the relevant vectors. Both the terms on the right side
of the last inequality can be controlled similarly so let us
demonstrate how to control the first term. Let $\theta^*_{1:m}$ have
$s_1$ constant pieces. Let us denote the blocks where $\theta^*_{1:m}$
is constant by $B_{1},B_{2},\dots,B_{s_1}$. Note that these blocks
necessarily are intervals. For any vector $a$ let $a_{B_{i}}$ denote
the $|B_{i}|$ dimensional vector which is $a$ restricted to the
coordinates in block $B_{i}$. Equipped with this notation we can now
write
\begin{align}
\sup_{\theta \in \M_{m}: \|\theta - \theta^*_{1:m}\| \leq t} \langle
z_{1:m}, \theta - \theta^*_{1:m} \rangle 
& \leq \sup_{\alpha \in \R_{+}^{s}: \|\alpha\| \leq t} \Bigl(\sum_{i = 1}^{s_1} \sup_{\theta \in \M_{|B_i|}: \|\theta_{B_i} - \theta^*_{B_i}\| \leq \alpha_i} \langle z_{B_i}, \theta_{B_i} - \theta^*_{B_i}\rangle\Bigr)\\
& = \sup_{\alpha \in \R_{+}^{s}: \|\alpha\| \leq t} \Bigl(\sum_{i = 1}^{s_1} \sup_{v \in \M_{|B_i|}: \|v\| \leq \alpha_i} \langle z_{B_i},v\rangle\Bigr)\\
\label{mid1}
&= \sup_{\alpha \in \R_{+}^{s}: \|\alpha\| \leq t} \Bigl(\sum_{i = 1}^{s_1} \alpha_i \sup_{v \in \M_{|B_i|}: \|v\| \leq 1} \langle z_{B_i},v\rangle\Bigr)
\end{align}
where the middle equality is due to the fact that 
$\theta^*$ is constant over the block $B_{i}$, and the 
second equality follows since the space of monotonic sequences is a cone.

For notational convenience we denote $\delta_i(z) = \sup_{v \in \M_{|B_i|}: \|v\| \leq 1} \langle z_{B_i},v\rangle$ for $1 \leq i \leq m$. Now define the event 
\begin{equation}\label{highp1}
\A_{m,1} = \bigcap_{i = 1}^{s_1} \bigl\{\delta_{i}(z) \leq \E \delta_{i}(z) + \sigma x\bigr\}
\end{equation}
for some $x > 0$ that will be determined below.
Note that $\delta_{i}(z)$ is a Lipschitz function of $z$ with Lipschitz constant one, as can be seen by applying Lemma~\ref{gaulip}. Therefore, using the Gaussian concentration result in Theorem~\ref{gaussconc} and a union bound we can conclude 
\begin{equation}\label{highp2}
P(\A_{m,1}) \geq 1 - s \exp(-x^2/2)
\end{equation}
where we have used $s_1 \leq s$. On the event $\A_{m,1}$, by~\eqref{mid1} we will then have
\begin{equation*}
\sup_{\theta \in \M_{m}: \|\theta - \theta^*_{1:m}\| \leq t} \langle z_{1:m}, \theta - \theta^*_{1:m} \rangle \leq \sup_{\alpha \in \R_{+}^{s}: \|\alpha\| \leq t} \sum_{i = 1}^{s_1} \alpha_i \big(\E \delta_{i}(z) + \sigma x\big).
\end{equation*} 
Applying the Cauchy-Schwarz inequality to the equation then gives us on the event $\A_{m,1}$,
\begin{equation*}
\sup_{\theta \in \M_{m}: \|\theta - \theta^*_{1:m}\| \leq t} \langle z_{1:m}, \theta - \theta^*_{1:m} \rangle \leq t \sqrt{\sum_{i = 1}^{s_1} \E \delta_i(z)^2} + t \sigma x \sqrt{s_1}. 
\end{equation*}
It can now be shown that $\E \delta_i(z)^2$ equals $\sigma^2$ times
the statistical dimension of the monotone cone of dimension
$|B_i|$. This equality is shown in the appendix in
Lemma~\ref{Statistical Dimension}. The statistical dimension of the
monotone cone in dimension $d$ is defined to be $\E \|\Pi_{\M} z\|^2$
where $z$ is a $d$-dimensional standard Gaussian vector with
independent entries and $\Pi_{\M}$ denotes the projection operator
onto the cone of nondecreasing sequences. Using known results
(see~\citet{bellec:2016}) for the statistical dimension of the
monotone cone then gives us
\begin{equation*}
\E \delta_i(z)^2 \leq \sigma^2 \log(e |B_i|)
\end{equation*}
from which we conclude
\begin{equation*}
\sup_{\theta \in \M_{m}: \|\theta - \theta^*_{1:m}\| \leq t} \langle z_{1:m}, \theta - \theta^*_{1:m} \rangle \leq t \sigma \sqrt{\sum_{i = 1}^{s_1} \log(e |B_i|)} + t \sigma x \sqrt{s_1}.
\end{equation*}
Since $\log$ is a concave function we can upper bound the term inside the square root by Jensen's inequality to write
\begin{align*}
\sum_{i = 1}^{s_1} \log(e |B_i|) = s_1 \sum_{i = 1}^{s_1} \frac{\log(e |B_i|)}{s_1} \leq s_1 \log\left(\frac{e}{s_1} \sum_{i = 1}^{s_1} |B_i|\right) \leq s \log\Bigl(\frac{en}{s}\Bigr)
\end{align*}
where we have used $m \leq n,s_1 \leq s$ and the fact that the function $g(x) = x \log(en/x)$ is increasing for $1 \leq x \leq n$.
Therefore, conditioned on the event $\A_{m,1}$ we have
\begin{equation*}
\sup_{\theta \in \M_{m}: \|\theta - \theta^*_{1:m}\| \leq t} \langle z_{1:m}, \theta - \theta^*_{1:m} \rangle \leq t \sigma \sqrt{s \log\Bigl(\frac{en}{s}\Bigr)} + t \sigma x \sqrt{s}.
\end{equation*}
It is crucial to note that since the event $\A_{m,1}$ does not depend
on $t$, the above inequality holds simultaneously for all $t \geq
0$. To control the second term in the right side
of~\eqref{firstsecond} one can define an analogous event $\A_{m,2}$
for which
\begin{equation}\label{highp2p}
\P(\A_{m,2}) \geq 1 - s \exp(-x^2/2).
\end{equation}
Therefore, by~\eqref{firstsecond} we can conclude that on the event $\A_{m,1} \cap \A_{m,2}$ we have for all $t \geq 0$,
\begin{equation*}
X_{m}(t) \leq 2 t \sigma \sqrt{s \log\Bigl(\frac{en}{s}\Bigr)} +  2 t \sigma x \sqrt{s}.
\end{equation*}
Since the right side of the above equation does not depend on $m$, we can now assert that on the event $\bigcap_{m = 1}^{n} \big(\A_{m,1} \cap \A_{m,2}\big)$ we have for all $t \geq 0$,
\begin{equation*}
\sup_{\theta \in \U_n:\|\theta - \theta^*\| \leq t} \langle z, \theta - \theta^* \rangle  \leq 2 t \sigma \sqrt{s \log\Bigl(\frac{en}{s}\Bigr)} +  2 t \sigma x \sqrt{s}.
\end{equation*}
Note that~\eqref{highp1} and~\eqref{highp2} imply, by a simple union
bound argument, that
\begin{equation*}
\P(\cap_{m = 1}^{n} \big(\A_{m,1} \cap \A_{m,2}\big)) \geq 1 - 2 n s \exp(-x^2/2) \geq 1 - 2 n^2 \exp(-x^2/2).
\end{equation*}
Setting $x^2 = 2 (\alpha + 2) \log n$ thus finishes the proof of the lemma.
\end{proof}

Using the above lemma, we can now complete the proof of Theorem 2.2.
\begin{proof}[Proof of Theorem 2.2]
Just as in the proof of Theorem 2.1, the first step here also is to reduce to the case when $\theta^*$ is monotonic. To do this, again let $1 \leq m^* \leq n$ be such that $\theta^* \in C_{m^*}$. Let
$\theta^* = (\theta^*_1,\theta^*_2)$ where $\theta^*_1$ is an $m^*$
dimensional vector and $\theta^*_2$ is an $n - m^*$ dimensional
vector. Break up $z$ similarly. Then we can write
\begin{equation}\label{asgoodasmono2}
\sup_{\theta \in \U_n: \|\theta - \theta^*\| \leq t} \langle z, \theta - \theta^* \rangle \leq \sup_{\theta \in \U_{m^*}: \|\theta - \theta^*_1\| \leq t} \langle z_1, \theta - \theta^*_1 \rangle + \sup_{\theta \in \U_{n - m^*}: \|\theta - \theta^*_2\| \leq t} \langle z_2, \theta - \theta^*_2 \rangle.
\end{equation}
Let us first bound the first term on the righthand side; the second
term can be bounded in exactly the same way. Let us denote $s_1$ to be the number of constant pieces of $\theta^*_1$ and $s_2$ to be the number of constant pieces of $\theta^*_2$. Since $\theta_1^{*}$ is nondecreasing, we can use Lemma~\ref{lemmaadap} to obtain with probability not less than $1 - 2 n^{-\alpha}$, simultaneously for all $t \geq 0$,
\begin{equation}
\sup_{\theta \in \M_{m^*}: \|\theta - \theta^*_1\| \leq t} \langle z_1, \theta - \theta^*_1 \rangle \leq 2 \sigma t \sqrt{s_1 \log\Bigl(\frac{e m^*}{s_1}\Bigr)} + 2 \sigma t \sqrt{2 s_1 (\alpha + 2)\log n}
\end{equation}
where we also use the fact that $(m^*) \leq n$. The second term in the right side of~\eqref{asgoodasmono2} can be upper bounded similarly. The last two displays then give us the upper bound, with probability not less than $1 - 4 n^{-\alpha}$,  for all $t \geq 0$ 
\begin{align*}\label{eqstep}
\sup_{\theta \in \U_{n}: \|\theta - \theta^*\| \leq t} \langle z, \theta - \theta^* \rangle\big) \leq &2 \sigma t \sqrt{s_1 \log\Bigl(\frac{e m^*}{s_1}\Bigr)} + 2 \sigma t \sqrt{s_2 \log\Bigl(\frac{e (n - m^*)}{s_2}\Bigr)} + \\&2 \sigma t \big(\sqrt{s_1} + \sqrt{s_2}\big) \sqrt{2 (\alpha + 2)\log n} 
\end{align*}
Recall now the definition of the function $f_{\theta^*}: \R_{+} \rightarrow \R$ as 
\begin{equation}\label{localgw}
f_{\bt^*}(t) = \sup_{\theta \in \U_{n}: \|\theta - \theta^{*}\| \leq t} \langle z,\theta - \theta^{*}\rangle - \frac{t^2}{2}.
\end{equation}
Then on an event of probability not less than $1 - 4 n^{-\alpha}$,  for all $t \geq 0$ we have
\begin{align*}
f_{\theta^*}(t) \leq &2 \sigma t \sqrt{s_1 \log\Bigl(\frac{e m^*}{s_1}\Bigr)} + 2 \sigma t \sqrt{s_2 \log\Bigl(\frac{e (n - m^*)}{s_2}\Bigl)} + 2 \sqrt{2} \sigma t \sqrt{2 (s_1 + s_2) (\alpha + 2)\log n} - \frac{t^2}{2} 
\end{align*}
where we have also used $\sqrt{s_1} + \sqrt{s_2} \leq \sqrt{2} \sqrt{s_1 + s_2}.$
Since the right side of the last display is a quadratic in $t$, it can
be verified that setting 
\begin{equation*}
t^* = 2 \sigma \sqrt{s_1 \log\Bigl(\frac{e m^*}{s_1}\Bigr)} + 2 \sigma \sqrt{s_2 \log\Bigl(\frac{e (n - m^*)}{s_2}\Bigr)} + 2 \sqrt{2} \sigma \sqrt{2 (s_1 + s_2) (\alpha + 2)\log n}
\end{equation*}
yields $f_{\theta^*}(t) < 0$ for all $t \geq t^*$ on an event of
probability not less than $1 - 4 n^{-\alpha}$. A bound on $(t^*)^2$
can then be obtained by using the elementary inequality $(a + b + c)^2
\leq 3 (a^2 + b^2 + c^2)$. The use of Lemma~\ref{chapro}
and an applications of Jensen's inequality in the form
\begin{equation*}
s_1 \log\left(\frac{e m^*}{s_1}\right) + s_2 \log\left(\frac{e (n -
  m^*)}{s_2}\right) 
\leq (s_1+s_2) \log\left(\frac{e n}{s_1+s_2}\right)
\end{equation*}
then finishes the proof of the theorem.
\end{proof}

\section{Discussion}


We have presented results on adaptivity of the least squares estimator
in unimodal sequence estimation. Our results show that the risk is no
worse than $O(n^{-2/3})$ uniformly over all unimodal sequences. More
interestingly, we have shown that for the same least squares
estimator, the risk decays no slower than $O\bigl(s(\theta^*)\log n/n)$ where $s(\theta^*)$ is the number of ``steps'' in the
true sequence. Thus, the estimator achieves nearly parametric rates of
convergence when the true sequence is simple.  Our proof techniques
exploit the structure of the space of unimodal sequences as a union of
$n$ convex cones (Lemma~\ref{chapro}), together with a peeling
argument that bounds the log-covering number of unimodal sequences in
the neighborhood of a monotone sequence (Lemma~\ref{key1}).

Lemma~\ref{chapro} (the ``LSE Slicing Lemma'') can potentially be used
in any estimation problem where the parameter space is a finite union
of convex cones and the true mean vector has disjoint pieces lying in
the intersection of these cones. It is the key ingredient used in our 
proof of the adaptive risk bound in Theorem~\ref{adaptive} for piecewise
constant unimodal sequences. Using an appropriate approximation 
of any unimodal sequence by a piecewise constant
unimodal sequence, Theorem~\ref{adaptive} and the slicing lemma could actually have been used
to prove Theorem~\ref{worstcase}, but we would have suffered an extra multiplicative log factor. 
We instead used a more direct peeling approach that resulted
in tighter bounds.

A natural problem for further study is the multidimensional setting,
observing that unimodality generalizes as quasiconvexity in dimensions
greater than one.  It is also important to study oracle inequalities,
similar to those obtained for isotonic regression by
\cite{chatterjee2015matrix} and \cite{bellec:2016}. As remarked in
Section~2, however, existing proof techniques exploit convexity
through bounds on Gaussian widths and KKT conditions.  Our
Lemma~\ref{chapro} that exploits the structure as a union of convex
cones indicates such results may indeed be achievable, but we believe
that new proof techniques may be required. Before submitting the final
version of this paper we became aware of the fact that such oracle
inequalities have recently appeared in~\cite{bellec2015sharp}
and~\cite{seriation}.

\section*{Acknowledgements} The first author would like to thank
Adityanand Guntuboyina and Bodhisattva Sen for helpful discussions. Supported in part by NSF
grant DMS-1513594 and ONR grant N00014-12-1-0762.

\vskip20pt
\setlength{\bibsep}{8pt}
\bibliographystyle{apalike}
\bibliography{unimodalrefer}

\appendix

\section{Appendix}

\begin{lemma}\label{subg}
Let $X_1,X_2\dots,X_n$ be random variables for which 
\begin{equation}\label{appe1}
\P\bigl(X_i \geq \E X_i + ax\bigr) \leq \exp\left(-\frac{x^2}{2}\right) \:\:\: \forall x \geq 0
\end{equation}
for every $1 \leq i \leq n$ and $a >0$.
Then 
\begin{equation}
\E\left( \max_{1 \leq i \leq n} X_i\right) \leq \max_{1 \leq i \leq n} \E X_i + a \left(\sqrt{2 \log n} + \sqrt{2 \pi}\right).
\end{equation}
\end{lemma}

\begin{proof}
Let $m = \max_{1 \leq i \leq n} \E X_i$. Define $Y_i = X_i - m$. Then by~\eqref{appe1}
we have sub-Gaussian tail behavior for every $1 \leq i \leq n$,
\begin{equation}\label{ybd}
\P(Y_i \geq x) \leq \exp\left(-\frac{x^2}{2 a^2}\right) \:\:\: \forall x \geq 0.
\end{equation}
Now by defining $Z_i = \max\{Y_i,0\}$ we certainly have $\E \max Y_i \leq \E \max Z_i$.
Also for any $x \geq 0$ we have $\P(\max Z_i \geq x) = \P(\max Y_i \geq x)$. Using the tail integral formula for the expectation of a nonnegative random variable we obtain
\begin{align}
\E \max Y_i &\leq \int_{0}^{\infty} \P(\max Y_i \geq x) \,dx \\
 & \leq \int_{0}^{\infty} \min\left\{n \exp\left(-\frac{x^2}{2
  a^2}\right),1\right\} dx \\
&= a \sqrt{2 \log n} + \int_{a \sqrt{2 \log n}}^{\infty} n \exp\left(-\frac{x^2}{2 a^2}\right) dx .
\end{align}
A standard fact about Gaussian tails give us the inequality
\begin{equation}
\int_{a \sqrt{2 \log n}}^{\infty} \:\:\exp\left(-\frac{x^2}{2 a^2}\right) \leq \sqrt{2 \pi} \frac{a}{n},
\end{equation}
which completes the proof of the lemma.
\end{proof}

\begin{lemma}\label{gaulip}
Let $A \subset \R^n$ be a closed set. Fix any $\theta^* \in \R^n$ and $t > 0$. Define the function $f:\R^n \rightarrow \R$ as follows:
\begin{equation}
f(z) = \sup_{\theta \in A: \|\theta - \theta^*\| \leq t} \langle z,\theta - \theta^* \rangle.
\end{equation}
Then $f$ is a Lipschitz function of $z$ with Lipschitz constant $t$.
\end{lemma}

\begin{proof}
Since the function $\langle z,\theta - \theta^* \rangle$ is a
continuous function of $z$ and the set $\{\theta \in A: \|\theta -
\theta^*\| \leq t\}$ is compact, the supremum is attained, at
$\tilde{\theta}$ say. Then we have
\begin{align}
f(z) - f(z') & = \langle z,\tilde{\theta} - \theta^* \rangle - f(z') \\
& \leq \langle z,\tilde{\theta} - \theta^* \rangle - \langle
z',\tilde{\theta} - \theta^* \rangle \\
& = \langle z - z',\tilde{\theta} - \theta^* \rangle \\
& \leq \|z - z'\| \|\tilde{\theta} - \theta^* \| \\
& \leq t \|z - z'\|.
\end{align}
The first inequality is because we set $\theta = \tilde{\theta}$
instead of taking supremum over $\theta$.  The second inequality is
just the Cauchy Schwarz inequality, and the last inequality follows
from the fact that $\|\tilde{\theta} - \theta^*\| \leq t$ by the
choice of $\tilde{\theta}$.
\end{proof}


\begin{lemma}\label{Statistical Dimension}
Let $z$ be a standard $d$ dimensional Gaussian vector with independent entries. Let $\Pi_{\M}$ denote the projection operator onto the space $\M \subset \R^d$. Then we have the following pointwise equality:
\begin{equation}
\|\Pi_{\M}(z)\| = \sup_{\theta \in \M: \|\theta\| \leq 1} \langle z, \theta \rangle.
\end{equation}
\end{lemma}

\begin{proof}
We can use Lemma 3.1 to prove this equality. Here $C = \M$ and hence
$C$ is a convex set. Take the zero vector in $C$. Then the least
squares estimator $\hat{\theta}$ is the same as $\Pi_{\M}(z)$ and Lemma 3.1 implies
\begin{equation}
\|\Pi_{\M}(z)\| = \argmax_{t \geq 0} \left(\sup_{\theta \in \M: \|\theta\| \leq t} \langle z, \theta \rangle - \frac{t^2}{2}\right).
\end{equation}
Now since the space $\M$ is a cone we can rewrite the last equation as 
 \begin{equation}
\|\Pi_{\M}(z)\| = \argmax_{t \geq 0}  \left(t \sup_{\theta \in \M: \|\theta\| \leq 1} \langle z, \theta \rangle - \frac{t^2}{2}\right).
\end{equation}
The right side above is the maximizer of a concave quadratic in
$t$; differentiating with respect to $t$ to compute the maximizer finishes the proof.
\end{proof}

\begin{lemma}\label{concav}
Let $g(t):\R_{+} \rightarrow \R$ be defined as 
$$g(t) = \sup_{\theta \in C \cap \{v \in \R^n: \|v\| \leq t\}} \langle a,\theta \rangle$$
where $C \subset \R^n$ is a closed convex set and $a$ is a fixed vector in $\R^n$. Then $g$ is a concave function.
\end{lemma}
\begin{proof}
We have to show that for any $t_1, t_2 > 0$ and $0 \leq \lambda \leq 1$ we have 
\begin{equation}
g(\lambda t_1 + (1 - \lambda) t_2) \geq \lambda g(t_1) + (1 - \lambda) g(t_2). 
\end{equation}
Let $\theta_1 \in C \cap \{v \in \R^n: \|v\| \leq t_1\}$ and $\theta_2 \in C \cap \{v \in \R^n: \|v\| \leq t_2\}$ be where the values $g(t_1)$ and $g(t_2)$ are achieved. Such a $\theta_1$ and $\theta_2$ exist because $C$ is closed and $\{v \in \R^n: \|v\| \leq t\}$ is closed and bounded. Now consider the point $\theta_3 = \lambda \theta_1 + (1 - \lambda) \theta_2$. Since $C$ is convex, $\theta_3 \in C$ and since the Euclidean squared norm is a convex function we also have $\|\theta_3\|  \leq t$. Now we have
\begin{equation}
\lambda g(t_1) + (1 - \lambda) g(t_2) = \langle a,\theta_3 \rangle \leq g(\lambda t_1 + (1 - \lambda) t_2)
\end{equation}
where the first equality is because of linearity of the inner product function and the last inequality is because of feasibility of $\theta_3$. This finishes the proof of the lemma.
\end{proof}

\end{document}